\documentclass{amsart}

\usepackage{amssymb,amsfonts,amsxtra}
\usepackage{dsfont,mathrsfs}
\renewcommand{\mathcal}{\mathscr}
\usepackage[all]{xypic}
\xyoption{dvips}
\usepackage{graphics}
\usepackage{url}
\usepackage[colorlinks,plainpages,backref]{hyperref}
\usepackage[neveradjust]{paralist}

\theoremstyle{definition}
\newtheorem{ntn}{Notation}[section]
\newtheorem{dfn}[ntn]{Definition}
\theoremstyle{plain}
\newtheorem{lem}[ntn]{Lemma}
\newtheorem{prp}[ntn]{Proposition}
\newtheorem{thm}[ntn]{Theorem}
\newtheorem{cor}[ntn]{Corollary}

\theoremstyle{remark}

\newtheorem{rmk}[ntn]{Remark}
\newtheorem{exa}[ntn]{Example}
\numberwithin{equation}{section}

\newcommand{\ideal}[1]{{\left\langle#1\right\rangle}}
\newcommand{\xymat}{\SelectTips{cm}{}\xymatrix}
\newcommand{\into}{\hookrightarrow}
\newcommand{\infrom}{\hookleftarrow}
\newcommand{\onto}{\twoheadrightarrow}

\newcommand{\llangle}{\langle\langle}
\newcommand{\rrangle}{\rangle\rangle}
\newcommand{\p}{\partial}
\newcommand{\prodcup}{\mathrel{\ooalign{$\cup$\cr\hss\scalebox{.75}{\raisebox{.5ex}{$\times$}}\hss}}}
\newcommand{\reg}{\mathrm{reg}}
\newcommand{\res}[2]{{\begin{bmatrix}#1\\#2\end{bmatrix}}}
\newcommand{\wh}{\widehat}
\newcommand{\wt}{\widetilde}
\newcommand{\ol}{\overline}
\newcommand{\ul}{\underline}

\renewcommand{\aa}{\mathfrak{a}}
\newcommand{\CC}{\mathds{C}}

\newcommand{\F}{\mathcal{F}}
\newcommand{\I}{\mathcal{I}}

\newcommand{\M}{\mathcal{M}}
\renewcommand{\O}{\mathcal{O}}

\newcommand{\mm}{\mathfrak{m}}
\newcommand{\NN}{\mathds{N}}
\newcommand{\pp}{\mathfrak{p}}
\newcommand{\qq}{\mathfrak{q}}

\DeclareMathOperator{\ann}{Ann}
\DeclareMathOperator{\Ass}{Ass}

\DeclareMathOperator{\coker}{coker}
\DeclareMathOperator{\depth}{depth}
\DeclareMathOperator{\Der}{Der}
\DeclareMathOperator{\edim}{edim}
\DeclareMathOperator{\End}{End}
\DeclareMathOperator{\Ext}{Ext}

\DeclareMathOperator{\grade}{grade}
\DeclareMathOperator{\height}{ht}
\DeclareMathOperator{\Hom}{Hom}

\DeclareMathOperator{\id}{id}
\DeclareMathOperator{\img}{im}
\DeclareMathOperator{\rk}{rk}

\DeclareMathOperator{\Spec}{Spec}
\DeclareMathOperator{\Supp}{Supp}

\DeclareMathOperator{\Tr}{Tr}

\begin{document}

\title{On Saito's normal crossing condition}

\author{Mathias Schulze}
\address{
M.~Schulze\\
Department of Mathematics\\
University of Kaiserslautern\\
67663 Kaiserslautern\\
Germany}
\email{mschulze@mathematik.uni-kl.de}
\thanks{The research leading to these results has received funding from the People Programme (Marie Curie Actions) of the European Union's Seventh Framework Programme (FP7/2007-2013) under REA grant agreement n\textsuperscript{o} PCIG12-GA-2012-334355.}

\date{\today}







\begin{abstract}
Kyoji Saito defined a residue map from the logarithmic differential $1$-forms along a reduced complex analytic hypersurface to the meromorphic functions on the hypersurface.
He studied the condition that the image of this map coincides with the weakly holomorphic functions, that is, with the functions on the normalization.
With Michel Granger, the author proved that this condition is equivalent to the hypersurface being normal crossing in codimension one.
In this article, the condition is given a natural interpretation in terms of regular differential forms beyond the hypersurface case.
For reduced equidimensional complex analytic spaces which are free in codimension one, the geometric interpretation of being normal crossing in codimension one is shown to persist.
\end{abstract}

\maketitle

\section*{Introduction}

Saito~\cite{Sai80} introduced the complex of logarithmic differential forms along a reduced hypersurface $D$ in a smooth complex manifold $S$.
It is defined as 
\[
\Omega^\bullet(\log D)=\{\omega\in\Omega_S^\bullet(D)\mid d\I_D\wedge\omega\subseteq\Omega^{\bullet+1}_S\}
\]
where $\I_D$ is the ideal sheaf of $D$.
Locally, if $\I_D=\ideal{h}$, such forms are characterized by having a presentation as
\[
g\omega=\frac{dh}{h}\wedge\xi+\eta
\]
where $\xi\in\Omega^{\bullet-1}_S$ and $\eta\in\Omega^{\bullet}_S$ have no pole and $g\in\O_S$ maps to a non-zero divisor in $\O_D$.
He defined a logarithmic residue map
\begin{equation}\label{6}
\rho_D\colon\Omega^\bullet(\log D)\to\M_D\otimes_{\O_D}\Omega_D^{\bullet-1},\quad\omega\mapsto\frac{\xi}{g}\vert_D
\end{equation}
where $\M_D=Q(\O_D)$ denotes the meromorphic functions on $D$.
This residue map gives rise to an exact sequence
\begin{equation}\label{7}
\xymat{
0\ar[r] & \Omega_S^\bullet\ar[r] & \Omega^\bullet(\log D)\ar[r]^-{\rho_D} & \sigma^{\bullet-1}_D\ar[r] & 0
}
\end{equation}
where $\sigma^{\bullet-1}_D$ denotes the image of $\rho_D$.
Let $\nu_D\colon\tilde D\to D$ be a normalization and note that $\M_D=\M_{\tilde D}$.
Saito~\cite[(2.8),(2.11)]{Sai80} showed that
\begin{equation}\label{13}
(\nu_D)_*\O_{\tilde D}\subseteq\sigma_D^0
\end{equation}
and that, if $D$ is a plane curve, equality holds if and only $D$ is normal crossing.
Generalizing this result to reduced hypersurfaces $D$, Granger and the author~\cite{GS14} showed that equality in \eqref{13} is equivalent to $D$ being normal crossing in codimension one.
The purpose of this article is to further generalize this preceding result.

In \S\ref{76}, we suggest a more general point of view for the equality in \eqref{13}.
It is based on Aleksandrov's result~\cite[\S4, Cor.~2]{Ale90} that $\sigma^\bullet_D=\omega_D^\bullet$ where the latter denotes the regular differential forms on $D$.
With Tsikh~\cite[Thm.~2.4]{AT01} (or \cite[Thm.~3.1]{AT08}) and later in \cite[Thm.~2]{Ale12} he generalized this result to complete intersections using (different versions of) multilogarithmic differential forms and their residues.
We relate it to Aleksandrov's multilogarithmic residue map and we comment on some claims made in \cite{Ale12}.
Regular differential forms are defined under more general hypotheses.
More specifically let $X$ be a reduced equidimensional complex analytic singularity with normalization $\nu_X\colon\tilde X\to X$.
Due to normality of $\tilde X$, we have $\O_{\tilde X}=\omega_{\tilde X}^0$ (see Corollary~\ref{45}).
We shall therefore refer to the equality
\[
(\nu_X)_*\omega_{\tilde X}^0=\omega_X^0
\]
resulting from \eqref{13} as \emph{Saito's normal crossing condition}.
Our approach is independent of an embedding and does not require a generalization of logarithmic differential forms such as multilogarithmic differential forms in the complete intersection case.
While Aleksandrov and Tsikh use Barlet's description of regular differential forms in the complex analytic context (see \cite{Bar78}) we prefer to rely on a general algebraic approach due to Kersken that is reviewed in \S\ref{32}.
In \S\ref{40} and \S\ref{75}, we study Saito's normal crossing condition for reduced curve and Gorenstein singularities. 
In \S\ref{36} we give it the following geometric interpretation analogous to \cite[Thm.~1.2]{GS14} in the hypersurface case.

\begin{thm}\label{70}
Let $X$ be a reduced equidimensional complex analytic singularity which is free in codimension one.
Then $X$ satisfies Saito's normal crossing condition if and only if $X$ is a normal crossing divisor in codimension one.\qed
\end{thm}

The additional freeness hypothesis replaces the fact that any reduced hypersurface is a free divisor in codimension one.
Our generalization of freeness is motivated by Aleksandrov--Terao theorem (see \cite[\S2 Thm.]{Ale88} and \cite[Prop.~2.4]{Ter80a}) stating that freeness of a reduced hypersurface is equivalent to Cohen--Macaulayness of the Jacobian ideal.
We call a reduced Gorenstein singularity free if the $\omega$-Jacobian ideal is a Cohen--Macaulay ideal (see Definition~\ref{98}).
In case of complete intersections of codimension $k$ Pol~\cite[Thm.~4.5]{Pol15} showed that freeness is equivalent to the projective dimension of multilogarithmic differential $k$ forms being equal to (or equivalently bounded by) $k-1$.
Her approach is a direct generalization of the one taken in \cite{GS14}.

\subsection*{Acknowledgments}

The author is grateful to Michel Granger and Delphine Pol for helpful comments.

\section{Regular and logarithmic differential forms}\label{32}

Fix a complete valued field $k$ of characteristic $0$ and let $A$ be a local analytic $k$-algebra of dimension $r\ge1$.
In particular $A$ is Noetherian, Henselian and catenary (see \cite[II.\S0.1,\S6.2]{GR71}).
Informally we refer to $A$ as a \emph{singularity}.

If $A$ admits a positive grading in the sense of Scheja and Wiebe (see \cite[\S3]{SW73}) then we call it a \emph{quasihomogeneous singularity}.
This means that $\mm_A$ is generated by eigenvectors of an \emph{Euler derivation} $\chi\in\Der_k(A,\mm_A)$ with positive rational eigenvalues $w_1,\dots,w_n$.
In this case one can write $\chi=\sum_{i=1}^nw_ix_i\p_{x_i}$.
If $w_1=\cdots=w_n$ then we call the grading a \emph{standard grading} and $A$ a \emph{homogeneous singularity}.

We denote by $Q(-)$ the total ring of fractions and abbreviate $L:=Q(A)$.
Let $R=k\llangle x_1,\dots,x_n\rrangle$ denote the regular ring of convergent power series over $k$ in $n$ variables $x_1,\dots,x_n$.
It is a formal power series ring in case the valuation is trivial.
For a suitable $n$, pick a finite $k$-algebra homomorphism
\begin{equation}\label{27}
R\to A
\end{equation}
of codimension $m=n-r$.

\subsection{Kersken's regular differential forms}\label{32a}

We begin by reviewing Kersken's description of regular differential forms (see \cite{Ker83a,Ker83b,Ker84}).
Denote by $\Omega_A:=\Omega_{A/k}$ the universally finite differential algebra of $A$ over $k$ (see \cite[\S11]{Kun86}).
In particular, $\Omega_A=\bigoplus_{p\in\NN}\Omega_A^p$ is graded with differential $d:\Omega_A\to\Omega_A[1]$ of degree $1$.
Let $C(A)$ be the (unaugmented) Cousin complex of $A$
\[
C(A)\colon 0\to C^0(A)\to C^1(A)\to\cdots
\]
with respect to $A$-active sequences (see~\cite[\S2]{Ker83a}).
It is a resolution of $A$ if and only if $A$ is Cohen--Macaulay and a (minimal) injective resolution if and only if $A$ is Gorenstein (see ~\cite{Sha69}).
Setting $C_\Omega(A):=C(A)\otimes_A\Omega_A$, the residue complex of $A$ is the complex of graded $(\Omega_A,d)$-modules
\[
D_\Omega(A):=\ul\Hom_{\Omega_R}(\Omega_A,C_\Omega(R))[m;m]
\]
where $\ul\Hom_{\Omega_R}$ denotes graded $\Hom_{\Omega_R}$ and $[m;m]$ signifies a shift by $m$ of both the $\Omega_R$-module and Cousin complex grading.
Notably this definition is independent of the choice of \eqref{27} (see \cite[(3.3)]{Ker83b}).
We write $\delta$ both for the Cousin differential of $C(R)$ and induced differentials.
The $0$th cohomology of $D_\Omega(A)$ with respect to $\delta$ is a graded $(\Omega_A,d)$-module
\[
\omega_A:=H^0(D_\Omega(A),\delta),
\]
the complex of \emph{regular differential forms} over $A$ (see \cite[p.~442]{Ker83b}).
For any graded $\Omega_R$-module $M$ one can identify (see~\cite[(3.6)]{Ker83b})
\begin{equation}\label{9}
\ul\Hom_{\Omega_R}(M,C_\Omega(R))=\Hom_R(M[n],\Omega_R^n\otimes_RC(R)).
\end{equation}
Since $C(R)$ is an injective resolution of $R$, this implies that $C_\Omega(R)$ is an injective resolution of $\Omega_R$.
It follows that (see~\cite[\S6]{Ker83b})
\[
\omega_A=\ul\Ext_{\Omega_R}^m(\Omega_A,\Omega_R)[m]
\]
which has graded components
\begin{equation}\label{10}
\omega^p_A=\Ext^m_R(\Omega_A^{r-p},\Omega_R^n)=\Hom_A(\Omega_A^{r-p},\omega_A^r)
\end{equation}
due to \eqref{9}, adjunction of $-\otimes_AA$ and $\Hom_R(A,-)$, and since $\Hom_R(A,C(R)^q)=0$ for $q<m$.

Kersken~\cite[\S5]{Ker83b} constructs a \emph{trace form}\footnote{Its construction uses that $k$ has characteristic $0$.} $c_A\in\omega_A^0$.
In case \eqref{27} is a \emph{Noether normalization} (see \cite[II.\S2.2]{GR71}), $c_A\in\omega_A^0=\ul\Hom_{\Omega_R}(\Omega_A,\Omega_R)$ restricts to (see \cite[(5.1.4)]{Ker83b})
\begin{equation}\label{82}
c_A\vert_{A\otimes_R\Omega_R}=\Tr_{A/R}\otimes_R\Omega_R\colon A\otimes_R\Omega_R\to\Omega_R
\end{equation}
where $\Tr_{A/R}\in\Hom_R(A,R)$ is the trace of $A$ over $R$ (see \cite[(10.3)]{SS74}).
It induces a unique \emph{trace map} of complexes of $(\Omega_A,d)$-modules (see \cite[(5.6)]{Ker83b})
\[
\gamma_A\colon C_\Omega(A)\to D_\Omega(A),\quad 1\mapsto c_A
\]
which is an isomorphism at regular primes of $A$ (see \cite[(5.7.2)]{Ker83b}).

If $A$ is reduced and equidimensional then  
\begin{equation}\label{91}
\xymat{
\Omega_A\otimes_A L=C^0_\Omega(A)\ar[r]^-{\gamma_A^0}_-\cong & D^0_\Omega(A)
}
\end{equation}
is an isomorphism.
It serves to identify $\omega_A$ with its preimage
\begin{equation}\label{81}
\sigma_A:=(\gamma_A^0)^{-1}(\omega_A),
\end{equation}
the complex of \emph{regular (meromorphic) differential forms} over $A$.
Under the identification \eqref{10} becomes
\begin{equation}\label{1}
\sigma^p_A=\Hom_A(\Omega_A^{r-p},\sigma_A^r).
\end{equation}
Composing $\Omega_A\to\Omega_A/T(\Omega_A)$ with $H^0(\gamma_A)$ yields a  map 
\begin{equation}\label{11}
c_A\colon\Omega_A\to\omega_A
\end{equation}
which is an isomorphism at regular primes of $A$ (see \cite[(5.7.3)]{Ker83b}).
We denote its cokernel by
\begin{equation}\label{55}
N_A:=\coker c_A.
\end{equation}
The preceding objects then fit into a commutative diagram
\begin{equation}\label{71}
\xymat{
&\Omega_A\otimes_A L\ar[r]^{\gamma_A^0}_\cong & D^0_\Omega(A)\\
&\sigma_A\ar[r]^\cong\ar@{^(->}[u] & \omega_A\ar@{^(->}[u]\\
\Omega_A\ar[uur]\ar[urr]_{c_A}\ar[ur]
}
\end{equation}
where the leftmost map is the canonical one.
In particular, its degree-$0$ part $A\into L$ factors through an inclusion
\begin{equation}\label{33}
c_A^0\colon A\into\sigma_A^0\cong\omega_A^0.
\end{equation}

If \eqref{27} is a presentation $R\onto A$ with kernel $\aa$ then (see \cite[Props.~3.8, 11.9]{Kun86})
\begin{equation}\label{90}
\Omega_A^p=\Omega_R^p/(\aa\Omega_R^p+d\aa\wedge\Omega_R^{p-1})=\bigwedge^p\Omega_A^1.
\end{equation}
In other words, $\Omega_A$ is an exterior differential algebra.
It follows that
\begin{equation}\label{17}
D_\Omega(A)=\ann_{C_\Omega(R)}(\aa\Omega_R+d\aa\wedge\Omega_R)[m;m].
\end{equation}
Elements of $C_\Omega(R)$ can be represented by residue symbols (see \cite[\S2]{Ker83a}), which lie by definition in the image of some map
\begin{equation}
\label{93}\Phi_{f_1,\dots,f_q}\colon(\Omega^p_R/\ideal{f_1,\dots,f_q}\Omega^p_R)_g\into C^q_\Omega(R),\quad\ol \xi/g\mapsto\res{\xi/g}{f_1,\dots,f_q},
\end{equation}
where $f_1,\dots,f_q,g$ is an $R$-sequence.
Injectivity of this map follows from \cite[(2.6)]{Ker83a} and Wiebe's Theorem~(see \cite[E.21]{Kun86}) using that the $\Omega_R^p$ are free $R$-modules.
The (induced) Cousin differential $\delta$ operates as (see \cite[(2.5)]{Ker83a}) 
\[
\delta\res{\xi/g}{f_1,\dots,f_q}=\res{\xi}{f_1,\dots,f_q,g}.
\]
Thus, elements of $\ker\delta$ are of the form $\res{\xi}{f_1,\dots,f_q}$ where $\ol\xi\in\Omega^p_R/\ideal{f_1,\dots,f_q}\Omega^p_R$.
One may assume that $f_1,\dots,f_m\in\aa$ after multiplying $\xi$ by a suitable transition determinant (see \cite[(2.5.3)]{Ker83a}).
Combined with \eqref{17} this yields the explicit description (see \cite[(1.2)]{Ker84}) 
\begin{align}\label{4}
\omega_A^p=\Bigl\{\res{\xi}{f_1,\dots,f_m}\Bigm\vert\ & \xi\in\Omega_R^{p+m},\ f_1,\dots,f_m\in\aa\ R\text{-sequence},\\
\nonumber & \aa\xi\equiv0\equiv d\aa\wedge\xi\mod\ideal{f_1,\dots,f_m}\Omega_R\Bigr\}.
\end{align}

\subsection{Aleksandrov's multilogarithmic residue}\label{32b}

In the following we describe Aleksandrov's generalization (see \cite{Ale12}) to complete intersections of \eqref{7} in relation with Kersken's description of regular differential forms in \S\ref{32a}.
To this end, consider $A=R/\aa$ with $\aa=\ideal{h_1,\dots,h_m}$ generated by an $R$-sequence $h_1,\dots,h_m$.
Then (see \cite[p.~445]{Ker83b})
\begin{equation}\label{60}
\gamma_A^q\colon\res{\ol\xi/\ol s}{\ol f_1,\dots,\ol f_q}\mapsto\res{d\ul h\wedge\xi/s}{\ul h,f_1,\dots,f_q}
\end{equation}
where $d\ul h:=dh_1\wedge\cdots\wedge dh_m$.
In particular, 
\begin{equation}\label{42}
c_A=\gamma_A^0(1)=\res{d\ul h}{\ul h}.
\end{equation}

The following types of differential forms with simple poles where introduced by Saito (see \cite{Sai80}) and implicitly by Aleksandrov (see \cite{Ale12}).
Notably the multilogarithmic differential forms of Aleksandrov and Tsikh (see \cite{AT01,AT08}) not considered here have arbitrary poles (see \cite[Appendix~B]{Pol15} for details).

\begin{dfn}\label{92}
Let $\ul h=h_1,\dots,h_m$ be an $R$-sequence and set $h:=h_1\cdots h_m$.
Then the \emph{logarithmic differential forms} along $\ideal{h}$ and the \emph{multilogarithmic differential forms} along $\ul h$ are defined respectively by
\begin{align*}
\Omega_R(\log\ideal{h})&:=\Bigl\{\omega\in\frac1h\Omega_R \Bigm\vert dh\wedge\omega\in\Omega_R\Bigr\},\\
\Omega_R(\log\ul h)&:=\Bigl\{\omega\in\frac1h\Omega_R\Bigm\vert\forall j=1,\dots,m\colon  dh_j\wedge\omega\in\sum_{i=1}^m\frac{h_i}{h}\Omega_R\Bigr\}.
\end{align*}
\end{dfn}

\begin{lem}\label{25}
Let $\ul h=h_1,\dots,h_m$ be an $R$-sequence. 
\begin{enumerate}[(a)]

\item\label{25a} An alternative definition of logarithmic differential forms reads
\begin{equation}\label{57}
\Omega_R(\log\ideal{h})=\Bigl\{\omega\in\frac1h\Omega_R\Bigm\vert\forall j=1,\dots,m\colon dh_j\wedge\omega\in\frac{h_j}h\Omega_R\Bigr\}.
\end{equation}
In particular, $\Omega_R(\log\ideal{h})\subseteq\Omega_R(\log\ul h)$ with equality for $m=1$.

\item\label{25b} There is an inclusion
\[
dh_i\wedge\Omega_R(\log\ideal{h})\subseteq\Omega_R(\log\ideal{h/h_i}).
\]

\item\label{25c} If $m\le2$ then
\[
\Omega_R(\log\ideal{h})\cap\sum_{i=1}^m\frac{h_i}h\Omega_R=\sum_{i=1}^m\Omega_R(\log\ideal{h/h_i}).
\]

\end{enumerate}
\end{lem}

\begin{proof}\
\begin{asparaenum}[(a)]

\item For $\omega\in\Omega_R(\log\ideal{h})$, we have 
\[
\sum_{i=1}^m\frac{h}{h_i}dh_i\wedge(h\omega)=hdh\wedge\omega\in h\Omega_R
\]
with $dh_i\wedge(h\omega)\in\Omega_R$.
Note that the factors $h_1,\dots,h_m$ of $h$ are pairwise coprime because they form an $R$-sequence.
It follows that $dh_i\wedge(h\omega)\in h_i\Omega_R$ for $i=1,\dots,m$.
Conversely, this latter condition implies that $dh\wedge\omega=\sum_{i=1}^m\frac {dh_i}{h_i}\wedge(h\omega)\in\Omega_R$.

\item For $\omega\in\Omega_R(\log\ideal{h})$, \eqref{25a} yields 
\[
dh_j\wedge dh_i\wedge\omega\in\frac{h_i}h\Omega_R\cap\frac{h_j}h\Omega_R=\frac{h_ih_j}{h}\Omega_R
\]
for $i\ne j$ and hence $dh_i\wedge\omega\in\Omega_R(\log\ideal{h/h_i})$.

\item Let $\sum_{i=1}^m\omega_i\in\Omega_R(\log\ideal{h})$ with $\omega_i\in\frac{h_i}h\Omega_R$ and set $\eta_i:=\frac h{h_i}\omega_{i}\in\Omega_R$.
By \eqref{25a} and \eqref{25b}, we have $dh_j\wedge\sum_{i\ne j}\omega_i\in\frac{h_j}h\Omega_R$ and hence $\sum_{i\ne j}h_idh_j\wedge\eta_i\in h_j\Omega_R$ for $j=1,\dots,m$.
Since $m\le2$ this implies that $dh_j\wedge\eta_i\in h_j\Omega_R$ and hence $dh_j\wedge\omega_i\in\frac{h_ih_j}{h}\Omega_R$ for $i\ne j$.
Thus, $\omega_i\in\Omega(\log\ideal{h/h_i})$ for $i=1,\dots,m$.\qedhere

\end{asparaenum}
\end{proof}

The following sequences appear in \cite[\S4, Lem.~1, \S6, Thm.~2]{Ale12}.

\begin{prp}\label{22}
Let $\ul h=h_1,\dots,h_m$ be an $R$-sequence.
Then there is a commutative diagram with exact top row (and exact bottom row if $m\le2$)
\begin{equation}\label{23}
\xymat{
0\ar[r]&\sum_{i=1}^m\frac{h_i}{h}\Omega_R\ar[r]&\Omega_R(\log\ul h)\ar[r]^-{\rho_{\ul h}}&\omega_A\ar[r]&0\\
0\ar[r]&\sum_{i=1}^m\Omega_R(\log\ideal{h/h_i})\ar@{^(->}[u]\ar[r]&\Omega_R(\log\ideal{h})\ar@{^(->}[u]\ar[r]^-{\rho_{\ul h}'}&\omega_A\ar@{=}[u]
}
\end{equation}
where $\rho_{\ul h}$ denotes the composition
\begin{equation}\label{2}
\xymat@R=0em{
\Omega_R(\log\ul h)\ar@{^(->}[r]^-{h\cdot} & \Omega_R\ar[r] & \Omega_R/\ideal{\ul h}\Omega_R\ar@{^(->}[r]^-{\Phi_{\ul h}} & \omega_A,\\
\omega=\frac\eta h\ar@{|->}[rrr] &&&\res{\eta}{\ul h}=z,
}
\end{equation}
with $\Phi_{\ul h}$ from \eqref{93}.
\end{prp}

\begin{proof}
By \eqref{4} and Definition~\ref{92} the map $\rho_{\ul h}$ is well-defined.
Using \cite[(2.5.3)]{Ker83a} and Wiebe's Theorem (see \cite[E.21]{Kun86}), any element of $\omega_A$ can be rewritten as in \eqref{4} with $f_1,\dots,f_m=\ul h$.
The vanishing conditions in \eqref{4} reduce to 
\[
dh_j\wedge\xi\equiv0\mod\ideal{\ul h}\Omega_R.
\]
Thus, the map $\rho_{\ul h}$ is surjective with kernel arising from the middle map in \eqref{2}.
If $m\le2$ then the left square in \eqref{23} is cartesian due to Lemma~\ref{25}.\eqref{25c}.
\end{proof}

We deduce the following characterization of multilogarithmic differential forms appearing in \cite[Thm.~1]{Ale12} (see also \cite[Prop.~2.1]{AT01} or \cite[Prop.~1.1]{AT08}).

\begin{cor}\label{35}
Let $\ul h=h_1,\dots,h_m$ be an $R$-sequence such that $A=R/\ideal{\ul h}$ is reduced.
For any $\omega\in\Omega_R(\log\ul h)$ there is a $g\in R$ with $\ol g\in A^\reg$, a $\xi\in\Omega_R$, and $\eta_i\in\frac{h_i}h\Omega_R$ for $i=1,\dots,m$, such that
\begin{equation}\label{34}
g\omega=\frac{d\ul h}{h}\wedge\xi+\sum_{i=1}^m\eta_i.
\end{equation}
Conversely, any $\omega\in\Omega_{R,h}$ admitting a representation \eqref{34} lies in $\Omega_R(\log\ul h)$.
\end{cor}

\begin{proof}
Let $\omega$ and $z$ be as in \eqref{2}.
By the isomorphism \eqref{91} and by \eqref{60}, there is a $g\in R$ and a $\xi\in\Omega_R$ as in the claim such that
\begin{equation}\label{20}
\rho_{\ul h}(g\omega)=
\res{g\eta}{\ul h}=
gz=
\gamma_A^0(\ol\xi)=
\res{d\ul h\wedge\xi}{\ul h}=
\rho_{\ul h}\left(\frac{d\ul h}{h}\wedge\xi\right).
\end{equation}
Then \eqref{34} follows from the exact sequence \eqref{23}.
Conversely let $\omega=\frac\eta h\in\Omega_{R,h}$ satisfy \eqref{34}.
Then $\eta\in\Omega_R$ with 
\[
g dh_j\wedge\eta=\sum_{i=1}^mdh_j\wedge(h\eta_i)\in\sum_{i=1}^mh_i\Omega_R
\]
and hence $dh_j\wedge\eta\in\sum_{i=1}^mh_i\Omega_R$ for $j=1,\dots,m$ since $h_1,\dots,h_m,g$ is an $R$-sequence.
It follows that $\omega\in\Omega_R(\log\ul h)$.
\end{proof}

\begin{rmk}\label{5}\
\begin{asparaenum}[(a)]

\item\label{5a} For $m=1$ the upper and lower sequences in \eqref{23} coincide by Definition~\ref{92}.

\item\label{5b} It follows from \eqref{34} and \eqref{20} that $(\gamma_A^0)^{-1}\circ\rho_{\ul h}$ coincides with Aleksandrov's multiple residue defined as in \eqref{6} (see \cite[\S4]{Ale12}).

\item\label{5c} Aleksandrov claims exactness of the bottom row for any $m$ and surjectivity of $\rho_{\ul h}'$ in \eqref{23} (see \cite[Thm.~2]{Ale12}).
However Pol showed that in general $\rho_{\ul h}'$ is not surjective (see \cite[Prop.~4.14]{Pol15}).

\end{asparaenum}
\end{rmk}

\section{Saito's normal crossing condition}\label{76}

In addition to the hypotheses of \S\ref{32} we shall assume from now on that $k$ is algebraically closed and that $A$ is $r$-equidimensional.
The integral closure of $A$ in $L=Q(A)$,
\begin{equation}\label{48}
\nu_A\colon A\into\tilde A,
\end{equation}
is a finite $k$-algebra homomorphism (see \cite[II.\S7.2]{GR71}),
the \emph{normalization} of $A$.
Denote by $\pp_1,\dots,\pp_s$ the minimal primes of $A$ and set 
\[
A_i:=A/\pp_i,\quad L_i:=Q(A_i).
\]
Then $\dim A_i=r$ by $r$-equidimensionality of $A$.
Since $A$ is reduced,
\begin{equation}\label{85}
\pp_iA_{\pp_i}=0,\quad L_i=A_{\pp_i}.
\end{equation}
For the same reason (see \cite[II.\S7.2]{GR71}),
\begin{equation}\label{29}
A\into\prod_{i=1}^sA_i\into\prod_{i=1}^s\tilde A_i=\tilde A\into\prod_{i=1}^sL_i=L
\end{equation}
where each $\tilde A_i=\wt{A_i}$ is a local analytic $k$-algebra.
Note that $L=Q(\tilde A)$ and $L_i=Q(\tilde A_i)$.
The objects of \S\ref{32a} can be defined verbatim for $\tilde A$ compatible with the product decomposition \eqref{29}.
In particular, $\gamma_{\tilde A}=\bigoplus_{i=1}^s\gamma_{\tilde A_i}$ and
\[
\omega_{\tilde A}=\bigoplus_{i=1}^s\omega_{\tilde A_i},\quad
\sigma_{\tilde A}=\bigoplus_{i=1}^s\sigma_{\tilde A_i}.
\]
For any $\qq\in\Spec\tilde A$ lying over $\pp=A\cap\qq\in\Spec A$, 
\begin{equation}\label{102}
\dim A_\pp=r-\dim A/\pp=r-\dim\tilde A/\qq=\dim\tilde A_{\qq}
\end{equation}
using that $A$ and $\tilde A$ are $r$-equidimensional and catenary (see \cite[Prop.~2.5.10]{Liu02})

\begin{prp}\label{12}
There is a commutative diagram
\begin{equation}\label{30}
\xymat{
\Omega_{\tilde A}\ar[r]^-{c_{\tilde A}} & \omega_{\tilde A}\ar@{^(->}[d] & \sigma_{\tilde A}\ar[l]^-\cong_{\gamma_{\tilde A}^0\vert}\ar@{^(->}[r]\ar@{^(->}[d] & \Omega_{\tilde A}\otimes_{\tilde A}L\\
\Omega_A\ar[r]^-{c_A}\ar[u]_{\bigwedge d\nu_A} & \omega_A & \sigma_A\ar[l]^-\cong_{\gamma_A^0\vert}\ar@{^(->}[r] & \Omega_A\otimes_AL\ar[u]_-\cong
}
\end{equation}
where the horizontal compositions are the canonical maps.
\end{prp}

\begin{proof}
Let \eqref{27} be a Noether normalization of $A$;
composed with \eqref{48} it gives a Noether normalization of $\tilde A$.
Setting $m=0$ in \eqref{10} it serves to compute both $\omega_A$ and $\omega_{\tilde A}$.
Note that $A\otimes_RQ(R)=L=\tilde A\otimes_RQ(R)$ and hence (see \cite[\S10]{SS74})
\begin{equation}\label{84}
\Tr_{\tilde A/R}\vert_A=\Tr_{A/R}.
\end{equation}
There is a natural map of complexes of graded $(\Omega_A,d)$-modules $D_\Omega(\tilde A)\to D_\Omega(A)$.
By \eqref{82} and \eqref{84} it maps $c_{\tilde A}\vert_{\tilde A\otimes_R\Omega_R}\mapsto c_A\vert_{A\otimes_R\Omega_R}$.
Together with the left claimed injectivity in diagram~\eqref{30} this implies that $c_{\tilde A}\mapsto c_A$ (see \cite[(5.1)]{Ker83a}).
The commutativity of diagram~\eqref{30} follows using diagram~\eqref{71}.

The inclusion \eqref{48} has torsion cokernel, so applying $\Hom_R(-,\Omega_R^n)$ first gives 
\begin{equation}\label{8}
\omega_{\tilde A}^r\into\omega_A^r
\end{equation}
due to \eqref{10}.
Consider the short exact sequence (see \cite[Cor.~11.8, Prop.~11.17]{Kun86})
\begin{equation}\label{95}
\xymat{
0\ar[r] & T^1(\tilde A/A)\ar[r] & \tilde A\otimes_A\Omega^1_A\ar[r]^-{\tilde A\otimes d\nu_A} & \Omega^1_{\tilde A}\ar[r] & \Omega^1_{\tilde A/A}\ar[r] & 0.
}
\end{equation}
Applying $\bigwedge^p$ to \eqref{95}, which is right-exact and commutes with base change, \eqref{90} gives a short exact sequence
\begin{equation}\label{96}
\xymat@C=3em{
0\ar[r] & T^p(\tilde A/A)\ar[r] & \tilde A\otimes_A\Omega^p_A\ar[r]^-{\tilde A\otimes\bigwedge^pd\nu_A} & \Omega^p_{\tilde A}\ar[r] & \Omega^p_{\tilde A/A}\ar[r] & 0
}
\end{equation}
where $T^p(\tilde A/A)$ is the image of $T^1(\tilde A/A)\otimes_A\Omega^{p-1}_A$ (see \cite[Prop.~A.2.2]{Eis95}).
Both $\Omega^1_A$ and $\Omega_{\tilde A}^1$ have rank $r$ (see \cite[(4.4)]{SS72}).
By finiteness of $\tilde A$ over $A$, $\Omega^1_{\tilde A/A}$ is the universal differential module which is compatible with localization and hence $\Omega^1_{\tilde A/A}\otimes_{\tilde A} L=0$.
It follows that $T^p(\tilde A/A)$ and $\Omega_{\tilde A/A}^p$ are torsion.
In particular, this gives the right vertical isomorphism in diagram~\eqref{30} and, since $\omega_{\tilde A}^r$ is torsion-free, we have
\begin{equation}\label{63}
\Hom_{\tilde A}(T^p(\tilde A/A),\omega_{\tilde A}^r)=0=\Hom_{\tilde A}(\Omega_{\tilde A/A}^p,\omega_{\tilde A}^r).
\end{equation}
Now \eqref{8} yields the upper inclusion and \eqref{96} and \eqref{63} the lower inclusion in the following diagram
\begin{equation}\label{31}
\xymat{
\omega_{\tilde A}^p\ar@{=}[d]&
\Hom_A(\Omega_A^{r-p},\omega_{\tilde A}^r)\ar@{=}[d]\ar@{^(->}[r]&
\Hom_A(\Omega_A^{r-p},\omega_A^r)\ar@{=}[d]\\
\Hom_{\tilde A}(\Omega_{\tilde A}^{r-p},\omega_{\tilde A}^r)\ar@{^(->}[r]&
\Hom_{\tilde A}(\tilde A\otimes_A\Omega_A^{r-p},\omega_{\tilde A}^r)&
\omega_A^p
}
\end{equation}
which proves injectivity of the vertical maps in diagram~\eqref{30}.
\end{proof}

The following fact stated by Kersken (see \cite[p.6]{Ker84}) goes back to a result of Serre (see \cite[p.~5]{Lip84}).

\begin{prp}\label{39}
If $A$ is normal then $\omega_A$ is a reflexive $A$-module.
\end{prp}

\begin{proof}
By Serre's criterion, normality of $A$ is equivalent to conditions $(R_1)$ and $(S_2)$.
Let \eqref{27} be a presentation $R\onto A$ with kernel $\aa$ and let $\qq\in\Spec A$.

First assume that $\depth A_\qq\le 1$.
Then $\dim A_\qq\le 1$ by $(S_2)$ and $A_\qq$ is regular by $(R_1)$.
It follows that \eqref{11} induces an isomorphism $\omega_{A,\qq}\cong\Omega_{A,\qq}$ and that $\Omega_{A,\qq}=\bigwedge\Omega_{A,\qq}^1$ is free (see \cite[(5.7.3)]{Ker83b} and \cite[(8.7)]{SS72}).
In particular, $\omega_{A,\qq}$ is reflexive in this case.

Then assume that $\depth A_\qq\ge 2$ and let $\pp\in\Spec R$ be the preimage of $\qq$.
Since $R$ is Cohen--Macaulay, $\grade(\aa,R)=m$ (see \cite[Thm.~2.1.2.(b)]{BH93}) and there is an $R$-sequence $\ul f=f_1,\dots,f_m\in\aa$.
Then $R_\pp/\ideal{\ul f}\onto A_\pp=A_\qq$ and since $R_\pp$ and hence $R_\pp/\ideal{\ul f}$ is Cohen--Macaulay (see \cite[Thm.~2.1.3.(a)]{BH93})
\[
\grade(\pp,R_\pp/\ideal{\ul f})=\dim(R_\pp/\ideal{\ul f})\ge\dim A_{\qq}\ge\depth A_\qq\ge2.
\]
Using $\Omega_A^0=A$ and $\Omega_R^n\cong R$ in \eqref{10}, $\omega^r_A\cong\Hom_R(A,R/\ideal{\ul f})$ (see \cite[Lem.~1.2.4]{BH93}).
It follows that (see \cite[Ex.~1.4.19]{BH93})
\[
\depth\omega^r_{A,\qq}=\grade(\qq,\omega^r_{A,\qq})=\grade(\pp,\Hom_{R_\pp}(A_\qq,R_\pp/\ideal{\ul f}))\ge2.
\]
Thus, reflexivity of $\omega_A^r$ and then of $\omega_A^p$ for all $p$ follows (see \cite[Prop.~1.4.1.(b)]{BH93}).
\end{proof}

\begin{cor}\label{45}
If $A$ is normal then $\sigma_A^0=\Omega_A^0=A$.
\end{cor}

\begin{proof}
Using \eqref{55} and \eqref{33} it suffices to show that $N_A^0=0$.
By hypothesis, $A$ satisfies Serre's conditions $(R_1)$ and $(S_2)$.
By $(R_1)$, $N_A^0$ has support of codimension at least $2$ (see \cite[(5.7.3)]{Ker83b}).
Let $\qq\in\Spec A$ with $\dim A_\qq\ge2$.
By $(S_2)$ and Proposition~\ref{39}, both $A_\qq$ and $\omega_{A,\qq}^0$ have depth at least $2$ (see \cite[Prop.~1.4.1.(b).(ii)]{BH93}).
Then $\depth N_{A,\qq}^0\ge1$ by the Depth Lemma (see \cite[Prop.~1.2.9]{BH93}) and hence $\qq\not\in\Ass N_A^0$.
Thus, $\Ass N_A^0=\emptyset$ and $N_A^0=0$ as claimed.
\end{proof}

In the hypersurface case, the inclusion $\omega_{\tilde A}^0\into\omega_A^0$ in diagram~\eqref{30} corresponds to the inclusion~\eqref{13} using Corollary \ref{45}.
This motivates the following 

\begin{dfn}\label{80}
We say that $A$ satisfies \emph{Saito's normal crossing condition (SNCC)} if $\omega_{\tilde A}^0=\omega_A^0$.
By \emph{SNCC at} $\pp\in\Spec A$ we mean that $\omega_{\tilde A,\pp}^0=\omega_{A,\pp}^0$.
\end{dfn}

We first note that SNCC is a codimension-one condition.

\begin{prp}\label{66}
The equality $\omega_{\tilde A}^p=\omega_A^p$ holds true if and only if it holds true in codimension one.
In particular, SNCC is a codimension-one condition.
\end{prp}

\begin{proof}
Assume that the inclusion $\omega_{\tilde A}^p\into\omega_A^p$ in diagram~\eqref{30} is an equality at primes of codimension $1$; denote by $W_A^p$ its cokernel.
Since $W_A^p$ is torsion, $W_A^p$ has support of codimension at least $2$.
Let $\pp\in\Spec A$ with $\dim A_\pp\ge2$ and pick any $\qq\in V(\pp\tilde A)\subseteq\Spec\tilde A$.
In particular, $\qq\cap A\supseteq\pp$ and hence $\dim\tilde A_\qq=\dim A_{\qq\cap A}\ge\dim A_\pp\ge2$ using \eqref{102}. 
By Serre's condition $(S_2)$ for $\tilde A$ then also $\depth\tilde A_\qq\ge2$.
Thus, $\depth\omega_{\tilde A,\qq}^p\ge 2$ by Proposition~\ref{39} (see \cite[Prop.~1.4.1.(b).(ii)]{BH93}).
It follows that (see \cite[IV.B.1.Prop.~12]{Ser65} and \cite[Prop.~1.2.10.(a)]{BH93})
\[
\depth\omega_{\tilde A,\pp}^p
=\grade(\pp,\omega_{\tilde A,\pp}^p)
=\grade(\pp\tilde A,\omega_{\tilde A,\pp}^p)
=\min\{\depth\omega_{\tilde A,\qq}^p\mid\qq\in V(\pp\tilde A)\}\ge 2.
\]
Since $\depth\omega_{A,\pp}^p\ge 1$ by diagram~\eqref{71}, the claim follows as in the proof of Corollary \ref{45}.
\end{proof}

Now we show that SNCC descends to any union of irreducible components.
For any subset $I\subseteq\{1,\dots,s\}$, set
\begin{equation}\label{89}
A_I:=A/\aa_I,\quad\aa_I:=\bigcap_{i\in I}\pp_i.
\end{equation}
Note that $A_I$ is reduced with minimal primes $\pp_i/\aa_I$, $i\in I$.

\begin{prp}\label{51}
If $\omega_{\tilde A}^p=\omega_A^p$ then $\omega_{\tilde A_I}^p=\omega_{A_I}^p$.
In particular, SNCC descends from $A$ to $A_I$ for any subset $I\subseteq\{1,\dots,s\}$.
\end{prp}

\begin{rmk}
Proposition~\ref{51} plays the role of the inclusion 
\[
\Omega^1(\log(D_1+D_2))\subseteq\Omega^1(\log D)
\]
for irreducible components $D_1$ and $D_2$ of a hypersurface $D$ used in \cite[Ex.~3.3]{GS14}.
\end{rmk}

The proof of Proposition~\ref{51} relies on the following two lemmas.

\begin{lem}\label{50}
For any subset $I\subseteq\{1,\dots,s\}$, we have $\omega_{A_I}^p=\Hom_A(\Omega_{A_I}^{r-p},\omega_A^r)$.
\end{lem}

\begin{proof}
Let \eqref{27} be a Noether normalization of $A$; composed with $A\onto A_I$ it gives a Noether normalization of $A_I$.
Using \eqref{10} and Hom-tensor-adjunction, we compute that
\[
\omega_{A_I}^r
=\Hom_R(A_I,\omega_R^r)
=\Hom_A(A_I,\Hom_R(A,\omega_R^r))
=\Hom_A(A_I,\omega_A^r)
\]
and hence that
\[
\omega_{A_I}^p
=\Hom_{A_I}(\Omega_{A_I}^{r-p},\omega_{A_I}^r)
=\Hom_{A_I}(\Omega_{A_I}^{r-p},\Hom_A(A_I,\omega_A^r))
=\Hom_A(\Omega_{A_I}^{r-p},\omega_A^r).\qedhere
\]
\end{proof}

Replacing $A$ in \eqref{89} by $\tilde A$, $\tilde\pp_j=\prod_{i\ne j}\tilde A_i$, $j=1,\dots,s$, are the minimal primes, $\tilde\aa_I=\prod_{i\not\in I}\tilde A_i$ and 
\[
\tilde A_I=\tilde A/\tilde\aa_I=\prod_{i\in I}\tilde A_i=\wt{A_I}.
\]

\begin{lem}\label{52}
The natural surjections $A_I\otimes_A\Omega_A^p\onto\Omega_{A_I}^p$ and $A_I\otimes_A\Omega_{\tilde A}^p\onto\Omega_{\tilde A_I}^p$ have torsion kernels $T^p(A_I/A)$ and $\tilde T^p(A_I/A)$, respectively.
\end{lem}

\begin{proof}
By definition, $T^0(A_I/A)=0$ and $\tilde T^0(A_I/A)$ is torsion by \eqref{85}.
In particular,
\begin{equation}\label{101}
A_I\otimes_A\Omega_{\tilde A}^p\onto\tilde A_I\otimes_{\tilde A}\Omega_{\tilde A}^p
\end{equation}
has torsion a kernel.
By \eqref{85}, $\aa_I/\aa_I^2$ is torsion and surjects onto $T^1(A_I/A)$ (see \cite[Cor.~11.10]{Kun86}).
Therefore $T^p(A_I/A)$ is torsion for all $p\ge1$ (see the proof of Proposition~\ref{12}).
Replacing $A$ by $\tilde A$ also $T^p(\tilde A_I/\tilde A)$ is torsion for all $p\ge1$.
By the Snake Lemma applied to 
\[
\xymat{
0\ar[r] & \tilde T^p(A_I/A)\ar[r]\ar[d] & A_I\otimes_A\Omega_{\tilde A}^p\ar[r]\ar@{->>}[d] & \Omega_{\tilde A_I}^p\ar[r]\ar@{=}[d] & 0\\
0\ar[r] & T^p(\tilde A_I/\tilde A)\ar[r] & \tilde A_I\otimes_{\tilde A}\Omega_{\tilde A}^p\ar[r] & \Omega_{\tilde A_I}^p\ar[r] & 0,
}
\]
$\tilde T^p(A_I/A)$ is an extension of the torsion kernel of \eqref{101} and $T^p(\tilde A_I/\tilde A)$.
\end{proof}

\begin{proof}[Proof of Proposition~\ref{51}]
Using \eqref{10}, Hom-tensor-adjunction, torsion-freeness of $\omega_A^r$, Lemmas~\ref{52} and \ref{50}, we compute
\begin{align}\label{99}
\Hom_A(A_I,\omega_A^p)
&=\Hom_A(A_I,\Hom_A(\Omega_A^{r-p},\omega_A^r))\\
\nonumber&=\Hom_A(A_I\otimes_A\Omega_A^{r-p},\omega_A^r)\\
\nonumber&=\Hom_A(\Omega_{A_I}^{r-p},\omega_A^r)
=\omega_{A_I}^p
\end{align}
and similarly $\Hom_A(A_I,\omega_{\tilde A}^p)=\omega_{\tilde A_I}^p$.
Thus, $\Hom_A(A_I,-)$ applied to the inclusion $\omega_{\tilde A}^p\into\omega_A^p$ in diagram~\eqref{30} yields the corresponding with $A$ replaced by $A_I$.
The claim follows.
\end{proof}

Finally, we show that SNCC is compatible with analytic triviality.

\begin{prp}\label{3}
Assume that $A=A'\hat\otimes R''$ where $A'$ satisfies the hypotheses on $A$, $\dim A'=r-1$ and $R''=k\llangle x\rrangle$ is regular.
Then $\omega_A^0=\omega_{A'}^0\hat\otimes R''$.
In particular, $A$ satisfies SNCC if and only if $A'$ does.
\end{prp}

\begin{proof}
Let \eqref{27} for $A'$ be a Noether normalization
\begin{equation}\label{86}
R'=k\llangle x_1,\dots,x_{r-1}\rrangle\into A'.
\end{equation}
A Noether normalization and a normalization of $A$ can be obtained by applying $-\hat\otimes R''$ to \eqref{86} and to \eqref{48} for $A'$ (see \cite[III.\S5]{GR71}), that is,
\[
R=R'\hat\otimes R''\into A=A'\hat\otimes R''\into\tilde A=\tilde A'\hat\otimes R''.
\]
This leads to decompositions (see \cite[III.\S5.10]{GR71})
\[
\Omega_R^r=\Omega_{R'}^{r-1}\hat\otimes\Omega_{R''}^1,\quad \Omega_A^r=\Omega_{A'}^r\hat\otimes R''\oplus\Omega_{A'}^{r-1}\hat\otimes\Omega_{R''}^1,
\]
where $\Omega_{R''}^1$ and $\Omega_R^r$ are free of rank $1$ and $\Omega_{A'}^r=\bigwedge^r\Omega_{A'}^1$ and hence $\Omega_{A'}^r\hat\otimes R''$ is torsion since $\rk\Omega_{A'}^1=\dim A'=r-1$ (see \cite[(8.8)]{SS72}).
Note that the analytic tensor products over $R',R''$ and over $A',R''$ coincide due to finiteness of $A'$ over $R'$ (see \cite[III.\S5.10]{GR71}).
Using \eqref{10} and flatness of $R'\to R$, we deduce
\begin{align*}
\omega_A^0&=\Hom_R(\Omega_A^r,\Omega_R^r)\\
&=\Hom_{R'\hat\otimes R''}(\Omega_{A'}^{r-1}\hat\otimes\Omega_{R''}^1,\Omega_{R'}^{r-1}\hat\otimes\Omega_{R''}^1)\\
&=\Hom_{R'\otimes_{R'} R\otimes_{R''}R''}(\Omega_{A'}^{r-1}\otimes_{R'}R\otimes_{R''}\Omega_{R''}^1,\Omega_{R'}^{r-1}\otimes_{R'}R\otimes_{R''}\Omega_{R''}^1)\\
&=\Hom_{R'\otimes_{R'} R}(\Omega_{A'}^{r-1}\otimes_{R'}R,\Omega_{R'}^{r-1}\otimes_{R'}R)\\
&=\Hom_{R'}(\Omega_{A'}^{r-1},\Omega_{R'}^{r-1}\otimes_{R'}R)\\
&=\Hom_{R'}(\Omega_{A'}^{r-1},\Omega_{R'}^{r-1})\otimes_{R'}R\otimes_{R''}R''\\
&=\omega_{A'}^0\hat\otimes R''
\end{align*}
and similarly $\omega_{\tilde A}^0=\omega_{\tilde A'}^0\hat\otimes R''$.
It follows that the inclusions $\omega_{\tilde A'}^0\into\omega_{A'}^0$ and $\omega_{\tilde A}^0\into\omega_A^0$ correspond via $-\hat\otimes R''$ and $-\otimes_{R''}k$.
\end{proof}

\section{Fractional ideals and ramification}\label{73}

Our approach to SNCC in case of curve and Gorenstein singularities uses that the inclusion $\omega_{\tilde A}^r\into\omega_A^r$ is given by the conductor ideal (see \eqref{41} and Lemma~\ref{49} below).
With the latter we recall the basics on fractional ideals.

\begin{dfn}
A (regular) \emph{fractional ideal} of $A$ is an $A$-submodule $M$ of $L=Q(A)$ such that there exist $a,b\in A^\reg$ with $aM\subseteq A$ and $b\in M$.
\end{dfn}

Since $A$ is Noetherian the first condition is equivalent to $M$ being finitely generated.
For any two fractional ideals $M,N\subset L$ of $A$ one can identify
\[
\Hom_A(M,N)= N:_LM\subseteq L,\quad \varphi\mapsto \frac{\varphi(m)}{m},\quad m\in M\cap A^\reg,
\]
with a fractional ideal of $A$.
The functor $\Hom_A(-,-)$ is inclusion-reversing (inclusion-preserving) in the first (second) argument on fractional ideals of $A$.
In particular, the dualizing operation
\[
-^{-1}:=\Hom(-,A)
\]
is inclusion-reversing on fractional ideals of $A$.
By \eqref{29}, $Q(A)_\pp=Q(A_\pp)$ and localization at $\pp\in\Spec A$ turns fractional ideals of $A$ into fractional ideals of $A_\pp$.
The localization of \eqref{48} at $\pp\in\Spec A$ is the normalization
\[
\nu_{A,\pp}\colon A_\pp\into\tilde A_\pp=\wt{A_\pp}
\]
of $A_\pp$ (see \cite[Prop.~2.1.6]{HS06}).
If $M$ is a fractional ideal of $A$ then
\[
\End_A(M)\subseteq\tilde A
\]
by the determinantal trick (see \cite[Lem.~2.1.8]{HS06}).
The \emph{conductor (ideal)}
\begin{equation}\label{37}
C_{\tilde A/A}:=\ann_A(\tilde A/A)=\tilde A^{-1}
\end{equation}
is the largest ideal of $A$ which is also an ideal of $\tilde A$.
Multiplying the denominators of a (finite) set of $A$-module generators of $\tilde A$ yields an element $b\in A^\reg\cap C_{\tilde A/A}$ showing that $C_{\tilde A/A}$ is a fractional ideal of $A$.

Both in case of curve and Gorenstein singularities the normalization will be unramified as a consequence of SNCC (see Propositions~\ref{15} and \ref{67} below).
Denote by $F^i_A(M)$ the $i$th \emph{Fitting ideal} of an $A$-module $M$.
Then the \emph{ramification ideal} of the normalization \eqref{48} is defined by
\[
I_{\tilde A/A}:=F^0_{\tilde A}(\Omega^1_{\tilde A/A}).
\]

\begin{lem}\label{44}
For any $\pp\in\Spec A$,
\[
(C_{\tilde A/A})_\pp=C_{\tilde A_\pp/A_\pp},\quad
(\Omega^1_{\tilde A/A})_\pp=\Omega^1_{\tilde A_\pp/A_\pp},\quad
(I_{\tilde A/A})_\pp=I_{\tilde A_\pp/A_\pp},
\]
and following statements are equivalent:
\begin{enumerate}[(a)]
\item $\tilde A_\pp$ is unramified over $A_\pp$.
\item $\Omega^1_{\tilde A_\pp/A_\pp}=0$.
\item $I_{\tilde A_\pp/A_\pp}=\tilde A_\pp$.
\end{enumerate}
In particular, $\Omega^1_{\tilde A/A}=0$ if and only if $A_i=\tilde A_i$ for $i=1,\dots,s$.
\end{lem}

\begin{proof}
By finiteness of $\tilde A$ over $A$, the conductor \eqref{37} commutes with flat base change and $\Omega^1_{\tilde A/A}$ is the universal differential module which commutes with base change.
Fitting ideals commute with flat base change.
The first claim and the equivalences follow (see \cite[Prop.~6.8]{Kun86}).
In particular, $\Omega^1_{\tilde A/A}=0$ if and only if $\tilde A$ is unramified over $A$.
Since $k=\ol k$, this is equivalent to 
\[
A_i/\mm_{A_i}=\tilde A_i/\mm_{\tilde A_i}=\tilde A_i/\mm_A\tilde A_i=\tilde A_i/\mm_{A_i}\tilde A_i
\]
and hence to $A_i=\tilde A_i$ for $i=1,\dots,s$ by Nakayama's Lemma.
\end{proof}

\section{Curve singularities}\label{40}

Keeping all hypotheses of \S\ref{76}, we assume in addition that $r=\dim A=1$.
Informally we refer to $A$ as a \emph{curve (singularity)} with \emph{branches} $A_1,\dots,A_s$ and we call it \emph{plane} if
\[
\edim A:=\dim_k(\mm_A/\mm_A^2)\le 2.
\]
By Serre's normality criterion, the $\tilde A_i$ in \eqref{29} are regular and hence (see \cite[II.\S5.3]{GR71})
\[
\tilde A_i=k\llangle t_i\rrangle.
\]
We denote by $e_1,\dots,e_s\in\tilde A$ the primitive idempotents with $\tilde Ae_i=\tilde A_i$.

For curve singularities we characterize SNCC numerically in terms of the De~Rham cohomology of $\omega_A$ and the \emph{$\delta$-invariant} of $A$
\[
\delta_A:=\dim_k(\tilde A/A).
\]

\begin{prp}
If $A$ is a curve singularity then
\[
\dim_kH^1(\omega_A)\le\delta_A
\]
with equality equivalent to SNCC.
\end{prp}

\begin{proof}
We set $\lambda_A:=\dim_k N_A^0$ (see \eqref{55}).
Then (see \cite[(4.5) Satz]{Ker84}),
\[
\dim_kH^1(\omega_A)=\mu_A-\lambda_A+s-1.
\]
Using Milnor's formula $\mu_A=2\delta_A-s+1$ (see \cite[Prop.~1.2.1.1)]{BG80}) this gives
\[
\dim_kH^1(\omega_A)=2\delta_A-\lambda_A.
\]
By Corollary~\ref{45}, the degree-$0$ part of the leftmost square in diagram~\eqref{30} reads
\[
\xymat{
\tilde A\ar@{=}[r]^-{c_{\tilde A}^0} & \omega_{\tilde A}^0\ar@{^(->}[d]\\
A\ar@{^(->}[r]^-{c_A^0}\ar@{^(->}[u] & \omega_A^0.
}
\]
Thus, $\lambda_A=\delta_A+\dim_k(\omega_A^0/\omega_{\tilde A}^0)$ and the claim follows.
\end{proof}

Our goal is to show that the only curve singularities satisfying SNCC are plane normal crossing.
For convenience we extend this notion as follows.
Denote the fiber product of the $\tilde A_i$ over $k$ by
\[
A\into\tilde A':=\tilde A_1\times_k\cdots\times_k\tilde A_s\into\tilde A.
\]

\begin{dfn}\label{59}
We call a curve singularity $A$ \emph{normal crossing} if $A=\tilde A'$.
\end{dfn}

If $A$ is normal crossing then $\mm_A=\mm_{\tilde A}$, $A_i=\tilde A_i$ for $i=1,\dots,s$, $\edim A=s$ and
\begin{equation}\label{43}
C_{\tilde A/A}=
\begin{cases}
A, & \text{if }s=1,\\
\mm_A=\mm_{\tilde A}, & \text{if }s\ge2.
\end{cases}
\end{equation}

We will first investigate the Gorenstein property of normal crossing curve singularities using the well-known results collected in the following lemma.
The statement on regularity goes back to Jacobinski in far greater generality (see \cite{Jac71}).

\begin{lem}\label{58}\ 
\begin{enumerate}[(a)]
\item\label{58a} $A\subseteq\mm_A^{-1}$ and, unless $A$ is regular, $\mm_A^{-1}\subseteq\tilde A$.
\item\label{58b} $A$ is Gorenstein if and only if $\dim_k(\mm_A^{-1}/A)=1$.
\end{enumerate}
\end{lem}

\begin{proof}\
\begin{asparaenum}[(a)]

\item If $\mm_A^{-1}\subsetneq\End_A(\mm_A)$ then there is a surjection $\mm_A\onto A$.
Since $A$ is projective it splits and hence $\mm_A=xA\oplus I$ for some $x\in A^\reg$
Then $xI\subseteq xA\cap I=0$ implies $I=0$.
It follows that $\mm_A=\ideal{x}$ and $A$ is regular.

\item Any $x\in\mm_A\cap A^\reg$ induces an isomorphism
\[
\xymat{
\Ext_A^1(k,A)\cong\Hom_A(k,A/xA)\cong(xA:_A\mm_A)/xA&\ar[l]^-\cong_-{\cdot x}\mm_A^{-1}/A.
}\qedhere
\]

\end{asparaenum}
\end{proof}

\begin{prp}\label{64}
A normal crossing curve singularity is Gorenstein if and only if it is plane.
\end{prp}

\begin{proof}
We may assume that $A$ is singular, that is, $s\ge2$.
By \eqref{43} and Lemma~\ref{58}.\eqref{58a}, $\mm_A^{-1}=\tilde A$ and hence
\[
\mm_A^{-1}/A\cong(\tilde A/\mm_{\tilde A})/(A/\mm_A)\cong k^s/k\cong k^{s-1}.
\]
By Lemma~\ref{58}.\eqref{58b}, $A$ is therefore Gorenstein if and only if $\edim A=s\le 2$.
\end{proof}

We now give a characterization of SNCC for curve singularities.
The proof relies on the identity (see \cite[Lem.~3.2]{KW84})
\begin{equation}\label{41}
\omega^r_{\tilde A}=C_{\tilde A/A}\omega^r_A.
\end{equation}
We abbreviate $\Der:=\Der_k$ to denote $k$-linear derivations.

\begin{prp}\label{15}
A curve singularity $A$ satisfies SNCC if and only if 
\begin{enumerate}[(a)]
\item\label{15a} $A$ has regular branches, that is, $A_i=\tilde A_i$ for $i=1,\dots,s$, and 
\item\label{15b} any $k$-derivation $A\to\omega^1_A$ factors through $\omega^1_{\tilde A}$, or equivalently,
\[
\Der(A)=\Der(A,C_{\tilde A/A})
\]
in case $A$ is Gorenstein.
\end{enumerate}
If $A$ is Gorenstein and singular then \eqref{15b} holds true if
\begin{equation}\label{21}
C_{\tilde A/A}=\mm_A
\end{equation}
and conversely \eqref{15b} implies \eqref{21} if in addition $A$ is quasihomogeneous.
\end{prp}

\begin{proof}
Recall from the proof of Proposition~\ref{12} that $T^1(\tilde A/A)$ and $\Omega_{\tilde A/A}^1$ in \eqref{95} are torsion.
So dualizing the short exact sequence
\[
0\to(\tilde A\otimes_A\Omega^1_A)/T^1(\tilde A/A)\to\Omega^1_{\tilde A}\to\Omega_{\tilde A/A}^1\to0
\]
obtained from \eqref{95} with the torsion-free module $\omega_{\tilde A}^1$ yields the following expansion of diagram~\eqref{31} in case $r=1$ and $p=0$.
\begin{equation}\label{14}
\xymat@C=10pt{
& 0\ar[r] & \Der(A,\omega^1_{\tilde A})\ar[r]\ar@{=}[d] & \Der(A,\omega^1_A)\ar@{=}[d]\\
& 0\ar[r] & \Hom_A(\Omega^1_A,\omega_{\tilde A}^1)\ar@{=}[dd]\ar[r] & \Hom_A(\Omega^1_A,\omega_A^1)\\
& \omega_{\tilde A}^0\ar@{=}[d] &&\omega_A^0\ar@{=}[u]\\
0\ar[r] & \Hom_{\tilde A}(\Omega^1_{\tilde A},\omega_{\tilde A}^1)\ar[r] & \Hom_{\tilde A}(\tilde A\otimes_A\Omega^1_A,\omega_{\tilde A}^1)\ar[r] & \Ext_{\tilde A}^1(\Omega_{\tilde A/A}^1,\omega_{\tilde A}^1)\ar[r] & 0
}
\end{equation}
The upper inclusion comes from the universal property of $\Omega_A^1$.
Its surjectivity is condition~\eqref{15b} and reads $\Der(A)=\Der(A,C_{\tilde A/A})$ for Gorenstein $A$ due to \eqref{41}.
Since $\omega_{\tilde A}^1$ is a canonical module of $\tilde A$ by \eqref{10} and $\Ext_{\tilde A}^1(\Omega_{\tilde A/A}^1,\omega_{\tilde A}^1)$ is the dual of $\Omega_{\tilde A/A}^1$ (see \cite[Thm.~3.3.10]{BH93}), surjectivity of the lower inclusion is equivalent to $\Omega_{\tilde A/A}^1=0$ and hence to condition~\eqref{15a} by Lemma~\ref{44}.
Therefore the diagram proves the first claim.

The remaining claims are due to the following facts.
If $A$ is singular then $C_{\tilde A/A}\subseteq\mm_A$ and $\Der(A)\subseteq\Der(A,\mm_A)$ (see \cite[(1.1)]{SW77}).
If $A$ is quasihomogeneous then $\chi(A)=\mm_A$ for some Euler derivation $\chi\in\Der(A,\mm_A)$ (see \cite{KR77} for a converse).
\end{proof}

\begin{rmk}
Let $A$ be a Gorenstein curve singularity.

\begin{asparaenum}[(a)]

\item Combining the degree-$0$ part of the leftmost square in diagram~\eqref{30} with diagram~\eqref{14} using \eqref{33} and \eqref{41} yields commutative diagram
\[
\xymat{
\tilde A\ar[r]^-{c_{\tilde A}^0} & \omega_{\tilde A}^0\ar@{^(->}[d]\ar@{^(->}[r] & \Der(A,C_{\tilde A/A})\ar@{^(->}[d]\\
A\ar@{^(->}[r]^-{c_A^0}\ar@{^(->}[u] & \omega_A^0\ar[r]_-\cong & \Der(A).
}
\]
The image of the bottom row is the module $\Delta$ of \emph{trivial derivations} (see \cite[\S3]{KW84} or \cite[\S5]{Ker84}).
Condition \eqref{15b} in Proposition~\ref{15} can therefore be rephrased as 
\[
\Der(A)/\Delta\to\Der(A,A/C_{\tilde A/A})
\]
being the zero map.

\item Proposition~\ref{51} can be deduced from Proposition~\ref{15} as follows.
It suffices to show that condition~\eqref{15b} in Proposition~\ref{15} descends from $A$ to $A_I$ for any subset $I\subseteq\{1,\dots,s\}$.
By \eqref{29}, there is a commutative diagram
\[
\xymat{
\tilde A\ar@{->>}[r]^{\tilde\pi_I}&\tilde A_I\\
A\ar@{^(->}[u]\ar@{->>}[r]^{\pi_I}&A_I\ar@{^(->}[u]
}
\]
and any $\delta_I\in\Der(A_I)$ lifts to a $\delta\in\Der(A)$ preserving $\aa_I$.
For $x_I\in A_I$, pick $x\in A$ with $\pi_I(x)=x_I$.
Assuming $\delta(x)\in C_{\tilde A/A}$, we compute using \ref{15}.\eqref{15b} for $A$ that
\[
\delta_I(x_I)\tilde A_I=\pi_I(\delta(x))\tilde\pi_I(\tilde A)=\tilde\pi_I(\delta(x)\tilde A)\subseteq\tilde\pi_I(A)=A_I
\]
and hence $\delta_I(x_I)\in C_{\tilde A_I/A_I}$ which is \ref{15}.\eqref{15b} for $A_I$.

\end{asparaenum}
\end{rmk}

We now examine SNCC for normal crossing curve singularities.

\begin{lem}\label{61}
A normal crossing curve singularity satisfies condition~\eqref{15b} of Proposition~\ref{15} if and only if it is plane.
\end{lem}

\begin{proof}
The canonical module $\omega_A^1$ of $A$ is an ideal (see \cite[Prop.~3.3.18]{BH93}).
With $A=\tilde A'$ also this ideal is standard graded and thus isomorphic to $A$ or to $\mm_A$.
Using Proposition~\ref{64}, \eqref{41} and \eqref{43}, this implies that
\[
\omega_A^1\cong
\begin{cases}
A, & \text{if }s\le 2,\\
\mm_A, & \text{if }s\ge3,
\end{cases}
\quad
\omega_{\tilde A}^1=
\begin{cases}
\omega_A^1, & \text{if }s=1,\\
\mm_A\omega_A^1, & \text{if }s\ge2.
\end{cases}
\]
If $A$ is singular then $\Der(A)\subseteq\Der(A,\mm_A)$ (see \cite[(1.1)]{SW77}) and $\chi(A)=\mm_A$ for some Euler derivation $\chi\in\Der(A,\mm_A)$.
Therefore condition~\eqref{15b} of Proposition~\ref{15} holds true if and only if $s\le 2$.
\end{proof}

Our starting point for understanding SNCC for general curve singularities are two examples that occur in the proof of the main theorem in \cite{GS14}.

\begin{exa}\label{18}\
\begin{asparaenum}[(a)]

\item\label{18a} In \cite[Ex.~3.3.(2)]{GS14}, $A$ is a plane quasihomogeneous curve defined by $\aa=\ideal{x_2(x_2-x_1^p)}$ where $p\ge1$.
Its normalization is given by $x_1=(t_1,t_2)$, $x_2=(0,t_2^p)$ and
\[
C_{\tilde A/A}=\ideal{(t_1^p,t_2^p)}=\ideal{x_1^p,x_2}.
\]
By Proposition~\ref{15}, $A$ satisfies SNCC if and only if $p=1$.

\item\label{18b} In \cite[Ex.~3.3.(3)]{GS14}, $A$ is the line arrangement defined by $\aa=\ideal{x_1x_2(x_1-x_2)}$.
Its normalization is given by $x_1=(t_1,0,t_3)$, $x_2=(0,t_2,t_3)$ and
\[
C_{\tilde A/A}=\ideal{(t_1^2,t_2^2,t_3^2)}=\ideal{x_1^2,x_2^2}.
\]
By Proposition~\ref{15}, SNCC does not hold.

\end{asparaenum}

Both statements above are shown in loc.~cit.~by a different argument due to Saito.
\end{exa}

Generalizations of Example~\ref{18} appear under the following conditions.

\begin{lem}\label{24}
Let $A$ be a non-normal crossing curve singularity different from that in Example~\ref{18}.\eqref{18a} with $s\ge2$ branches.
Assume that $A_{I}$ is normal crossing for all $I\subset\{1,\dots,s\}$ with $|I|=s-1$.
Then $A$ is the union of $s-1\ge2$ coordinate axes and a diagonal as defined by \eqref{24b}.
In particular, $A$ is homogeneous and Gorenstein of embedding dimension $n=\edim A=s-1$ with conductor $C_{\tilde A/A}=\mm_A^2$.
\end{lem}

\begin{proof}
With $s\ge2$ also $n\ge2$ and $A_i=\tilde A_i$ for $i=1,\dots,s$.
Set $J:=\{1,\dots,s-1\}$.
Then $A_J$ is normal crossing but $A$ is not.
Thus, there is a commutative diagram with exact rows
\[
\xymat{
0\ar[r] & \mm_{A_s}\ar[r] & \tilde A'\ar[r] & \tilde A'_J\ar[r] & 0\\
0\ar[r] & \aa_J\ar@{_(->}[u]\ar[r] & A\ar@{_(->}[u]\ar[r] & A_J\ar@{=}[u]\ar[r] & 0
}
\]
in which the leftmost inclusion is strict.
For any $j\in J$, both $A_J$ and $A_{\{1,\dots,s\}\setminus\{j\}}$ are normal crossing.
So there is an element $x_j\in\mm_A$ inducing uniformizers of $A_j$ and $A_s$ but zero in $\mm_{A_i}$ for any $i\ne j,s$.
Additional generators of $A$ can be chosen from $\aa_J\subseteq\mm_{A_s}^2$.
The inclusion $A\subseteq\tilde A'$ is then given by
\begin{equation}\label{47}
x_i=
\begin{cases}
u_it_ie_i+v_it_se_s, & i=1,\dots,s-1,\\
w_it_s^{p_i}e_s, & i=s,\dots,n,
\end{cases}
\end{equation}
where the $u_i\in A_i^*$ and the $v_i,w_i\in A_s^*$ are units, $p_i\ge2$, and $n\ge s-1$.
If $n\ge s$, we may assume that $p:=p_s$ is minimal and replace $t_s$ to absorb $w_s$.
For $i<s$, we replace $x_i$ and $t_i$ to absorb $v_i$ and $u_i$.
For $i>s$ and $j<s$, we have
\[
x_i=w_it_s^{p_i}e_s=w_it_s^{p_i-p}t_s^pe_s=w_i(t_je_j+t_se_s)(t_je_j+t_se_s)^{p_i-p}t_s^pe_s=w_i(x_j)x_j^{p_i-p}x_s
\]
which makes $x_i$ redundant.

So we may finally assume that $u_i=v_i=w_i=1$ and $n\le s$ in \eqref{47}.
This leaves the following two cases extending Example~\ref{18}.
\begin{gather}
\label{24a}
n=s\ge2,\quad p\ge2,\quad x_i=
\begin{cases}
t_ie_i+t_se_s, & i=1,\dots,s-1,\\
t_s^pe_s, & i=n,
\end{cases}\\
\label{24b}
n=s-1\ge2,\quad x_i=t_ie_i+t_se_s,\ i=1,\dots,n.
\end{gather}
For $n=2$, \eqref{24a} and \eqref{24b} define the curve singularities from parts \eqref{18a} and \eqref{18b} of Example~\ref{18}, respectively.
For $n\ge3$, \eqref{24a} reduces to \eqref{24b} since $x_n=x_1x_2^{p-1}$ is redundant.
Then Lemma~\ref{83} below concludes the proof.
\end{proof}

\begin{lem}\label{83}
The curve singularity $A$ defined by \eqref{24b} is homogeneous and Gorenstein.
\end{lem}

\begin{proof}
It follows from \eqref{24b} that $A=R/\aa$ is defined by $\aa=\ideal{x_k(x_i-x_j)\mid k\ne i,j}$, and hence homogeneous, and that the conductor equals $C_{\tilde A/A}=\mm_{\tilde A}^2=\mm_A^2$.
By Lemma~\ref{58}.\eqref{58a}, $\mm_A^{-1}/A$ can be seen as a subquotient in
\[
\mm_{\tilde A}^2=C_{\tilde A/A}\subseteq A\subseteq\mm_A^{-1}\subseteq\tilde A.
\]
Due to homogeneity of $A$ this is a chain of standard graded ideals.
Then with $\tilde A/\mm_{\tilde A}^2$ also $\mm_A^{-1}/A$ is non trivial at most in degrees $0$ and $1$.
It follows from \eqref{24b} that $\mm_A^{-1}$ and $A$ have equal constant parts.
Setting $t:=\sum_{i=1}^st_ie_i$, we have $t\cdot x_i=x_i^2\in A$ for $i=1,\dots,n$ and hence $t\in\mm_A^{-1}\setminus A$.
On the other hand, $x_1,\dots,x_n,t$ is a $k$-basis of the linear part of $\tilde A$ with $x_1,\dots,x_n\in A$.
Thus, $t$ represents a $k$-basis of $\mm_A^{-1}/A$ and $A$ is Gorenstein by Lemma~\ref{58}.\eqref{58b}.
\end{proof}

We can finally show that SNCC characterizes plane normal crossings among all curve singularities.

\begin{prp}\label{97}
A curve singularity satisfies SNCC if and only if it is plane normal crossing.
\end{prp}

\begin{proof}
Plane normal crossing curve singularities are Gorenstein and therefore satisfy SNCC by \eqref{43} and Proposition~\ref{15}.
Conversely, let $A$ be a curve singularity with $s$ branches satisfying SNCC.
If $s=1$ then $A=A_1=\tilde A_1=\tilde A'$ by Proposition~\ref{15}.\eqref{15a}.
We now proceed by induction on $s$ assuming $s\ge2$.
Due to Proposition~\ref{51} and the induction hypothesis, $A_{I}$ is normal crossing for all $I\subset\{1,\dots,s\}$ with $|I|=s-1$.
The only curve singularity in Example~\ref{18}.\eqref{18a} satisfying SNCC is plane normal crossing.
The conclusion of Lemma~\ref{24} contradicts to Proposition~\ref{15}.
Therefore $A$ must be normal crossing and hence plane by Lemma~\ref{61}.
\end{proof}

\section{Gorenstein singularities}\label{75}

Keeping all hypotheses of \S\ref{76}, we assume in addition that $A$ is Cohen--Macaulay and Gorenstein at $\pp\in\Spec A$.
By \eqref{10}, $\omega^r_A$ is then a canonical module of $A$ and hence (see \cite[Thms.~3.3.5.(b), 3.3.7]{BH93})
\begin{equation}\label{112}
\omega^r_{A,\pp}=\omega^r_{A_\pp}\cong A_\pp.
\end{equation}
In particular, $-^{-1}:=\Hom_{A_\pp}(-,A_\pp)$ corresponds to the duality $\Hom_{A_\pp}(-,\omega_{A,\pp}^r)$ on maximal Cohen--Macaulay modules.

\begin{lem}\label{49}
Let $A$ be Cohen--Macaulay and Gorenstein at $\pp\in\Spec A$. 
Then
\[
\omega^r_{\tilde A,\pp}=C_{\tilde A_\pp/A_\pp}\omega^r_{A,\pp}\cong C_{\tilde A_\pp/A_\pp}.
\]
\end{lem}

\begin{proof}
Let \eqref{27} be a Noether normalization.
By \eqref{10} and Hom-tensor-adjunction, 
\[
\omega^r_{\tilde A}=
\Hom_R(\tilde A,\Omega^r_R)=
\Hom_A(\tilde A,\Hom_R(A,\Omega^r_R))=
\Hom_A(\tilde A,\omega^r_A).
\]
By finiteness of $\tilde A$ over $A$ and \eqref{112}, localization at $\pp$ turns this into
\[
\omega^r_{\tilde A,\pp}=
\Hom_{A_\pp}(\tilde A_\pp,\omega^r_{A,\pp})=
\Hom_{A_\pp}(\tilde A_\pp,A_\pp)\omega^r_{A,\pp}=
C_{\tilde A_\pp/A_\pp}\omega^r_{A,\pp}.\qedhere
\]
\end{proof}

\begin{dfn}
The \emph{Jacobian} and \emph{$\omega$-Jacobian (ideal)} of $A$ are defined by 
\begin{equation}\label{72}
J_A:=F^r_A(\Omega^1_A),\quad J'_A:=\ann\coker c_A^r=\img(c_A^r\otimes(\omega_A^r)^{-1}).
\end{equation}
\end{dfn}

The ideals in \eqref{72} satisfy inclusion relations (see \cite[Prop.~3.1]{OZ87})
\begin{equation}\label{74}
J_A\subseteq J_A'\subseteq C_{\tilde A/A}.
\end{equation}
The second inclusion is due to Lemma~\ref{49} and the degree-$r$ part of the leftmost square in diagram~\eqref{30}.

\begin{rmk}\label{103}
Since $\Omega^1_A$ has rank $r$ (see \cite[(4.4)]{SS72}), $J_{A,\pp_i}=F^r_{A_{\pp_i}}(\Omega^1_{A,\pp_i})=A_{\pp_i}$ for $i=1,\dots,s$ and $J_A$ contains a regular element of $A$ by prime avoidance.
It follows that both $J_A$ and $J_A'$ are fractional ideals of $A$.
In case of $J_A'$ this follows also from $c_A$ being an isomorphism at regular primes of $A$ (see \cite[(5.7.3)]{Ker83b}) and Serre's reducedness criterion.
If $A$ is a complete intersection then $J_A=J_A'$ (see \cite[Lem.~3.1]{SS79} or \cite[Prop.~1]{Pie79} and \cite[Prop.~3.2]{OZ87} for a converse).
\end{rmk}

The statement of \cite[Prop.~3.4]{GS14} for hypersurface singularities generalizes by replacing the Jacobian by the $\omega$-Jacobian.

\begin{lem}\label{46}
Let $A$ be Cohen--Macaulay and Gorenstein at $\pp\in\Spec A$.
Then
\[
\sigma_{A,\pp}^0=(J_{A,\pp}')^{-1}
\]
as fractional ideals of $A_\pp$.
\end{lem}

\begin{proof}
We use \eqref{81} to identify $\omega_A$ with $\sigma_A$.
By \eqref{72} and the Gorenstein hypothesis this turns $c_{A,\pp}^r$ into a map $\Omega_{A,\pp}^r\onto J_{A,\pp}'\sigma_{A,\pp}^r$ with torsion cokernel. 
Then \eqref{1} localized at $\pp$ becomes $\sigma_{A,\pp}^0=\Hom_{A,\pp}(J_{A,\pp}'\sigma_{A,\pp}^r,\sigma_{A,\pp}^r)=(J_{A,\pp}')^{-1}$.
\end{proof}

\begin{dfn}\label{100}
We call $A$ \emph{free at} $\pp\in\Spec A$ if $A$ is Cohen--Macaulay, $A_\pp$ is Gorenstein and $J'_{A,\pp}$ is a Cohen--Macaulay ideal.
We say that $A$ is \emph{free} if it is free at $\mm_A$.
\end{dfn}

The Aleksandrov--Terao theorem (see \cite[\S2 Thm.]{Ale88} and \cite[Prop.~2.4]{Ter80a}) generalizes as follows.

\begin{prp}\label{69}
Let $A$ be Cohen--Macaulay and Gorenstein at $\pp\in\Spec A$.
Then freeness of $A$ at $\pp$ with $A_\pp\ne J_{A,\pp}'$ is equivalent to $A_\pp/J_{A,\pp}'$ being Cohen--Macaulay of dimension $\dim A_\pp-1$.
\end{prp}

\begin{proof}
By Remark~\ref{103}, $J_{A,\pp}'\subsetneq A_\pp$ is a fractional ideal of $A_\pp$ (see \S\ref{73}).
In particular, it contains an element of $A_\pp^\reg\setminus A_\pp^*$ and hence $\height J_{A,\pp}'\ge 1$.
The claim follows (see \cite[Satz~4.13]{HK71} and \cite[Thm.~2.1.2.(a)]{BH93}).
\end{proof}

By \eqref{74}, \eqref{37}, Corollary~\ref{45}, and Propositions~\ref{12}, there is an ascending chain of fractional ideals 
\begin{equation}\label{87}
J_A'\subseteq C_{\tilde A/A}\subseteq A\subseteq\tilde A=\sigma_{\tilde A}^0\subseteq\sigma_A^0.
\end{equation}

We deduce the following generalization of \cite[Cor.~3.7]{GS14}.

\begin{cor}\label{62}
Let $A$ be Cohen-Macaulay and free at $\pp\in\Spec A$.
Then $A$ satisfies SNCC at $\pp$ if and only if $J'_{A,\pp}=C_{\tilde A_\pp/A_\pp}$.
\end{cor}

\begin{proof}
By reflexivity of $\tilde A$ (see \cite[Lem.~2.8]{GS14}), \eqref{37} and Lemma~\ref{46}, the first and last inclusions in \eqref{87} localized at $\pp\in\Spec A$ are duals of each other.
\end{proof}

We recall an identity of ideals due to Piene (see \cite[Cor.~1]{Pie79}) in case of a smooth normalization.

\begin{lem}\label{94}
Let $A$ be Cohen--Macaulay and let $\pp\in\Spec A$ such that $A_\pp$ is Gorenstein and $\tilde A_\pp$ is regular.
Then $I_{\tilde A_\pp/A_\pp}C_{\tilde A_\pp/A_\pp}=\tilde AJ_{A,\pp}'$.
\end{lem}

\begin{proof}
Since $\tilde A_\pp$ is regular, $\Omega^1_{\tilde A,\pp}$ is locally free of rank $r$ (see \cite[(4.4),(8.7)]{SS72}).
The map $\tilde A\otimes d\nu_A$ from \eqref{95} is a presentation of $\Omega^1_{\tilde A/A}$.
Using Lemma~\ref{44}, it follows that $I_{\tilde A_\pp/A_\pp}\Omega^r_{\tilde A,\pp}$ is the image of the map
\[
\tilde A\otimes\bigwedge^rd\nu_{A,\pp}\colon\tilde A\otimes_A\Omega^r_{A,\pp}\to\Omega^r_{\tilde A,\pp}
\]
obtained by localizing the map $\tilde A\otimes\bigwedge^rd\nu_A$ from \eqref{96} at $\pp$.
Together with the degree-$r$ part of the leftmost square in diagram~\eqref{30} localized at $\pp$ this map fits into a commutative diagram
\[
\xymat{
&\Omega_{\tilde A,\pp}^r\ar[r]^-{c_{\tilde A,\pp}^r}_-\cong & \omega_{\tilde A,\pp}^r\ar@{^(->}[d]\\
&\Omega_{A,\pp}^r\ar[dl]\ar[r]^-{c_{A,\pp}^r}\ar[u]_{\bigwedge^rd\nu_{A,\pp}} & \omega_{A,\pp}^r\\
\tilde A\otimes_A\Omega_{A,\pp}^r\ar[uur]^{\tilde A\otimes\bigwedge^rd\nu_{A,\pp}}\ar[urr]_{\tilde A\otimes c_{A,\pp}^r}
}
\]
where $c_{\tilde A,\pp}^r$ is an isomorphism since $\tilde A_\pp$ is regular (see \cite[(5.7.3)]{Ker83b}).
Using Lemma~\ref{49} and \eqref{112} it follows that 
\begin{align*}
I_{\tilde A_\pp/A_\pp}C_{\tilde A_\pp/A_\pp}\omega^r_{A,\pp}
=I_{\tilde A_\pp/A_\pp}\omega^r_{\tilde A,\pp}
&=\img\Bigl(c_{\tilde A,\pp}^r\circ\tilde A\otimes\bigwedge^rd\nu_{A,\pp}\Bigr)\\
&=\img(\tilde A\otimes c_{A,\pp}^r)
=\tilde A\img c_{A,\pp}^r
=\tilde AJ_{A,\pp}'\omega^r_{A,\pp}.
\end{align*}
The claim follows by \eqref{112}.
\end{proof}

The following result generalizes \cite[Lem.~4.2]{GS14}. 

\begin{prp}\label{67}
Let $A$ be Cohen--Macaulay and free at $\pp\in\Spec A$ such that $\tilde A_\pp$ is regular.
Then $A$ satisfies SNCC at $\pp$ if and only if $J'_{A,\pp}$ is an ideal of $\tilde A_\pp$ and $\tilde A_\pp$ is unramified over $A_\pp$.
\end{prp}

\begin{proof}
By Lemma~\ref{49}, \eqref{90} and regularity of $\tilde A_\pp$ (see \cite[(5.7.3)]{Ker83b}), 
\[
C_{\tilde A_\pp/A_\pp}\cong\omega_{\tilde A,\pp}^r\cong\Omega_{\tilde A,\pp}^r=\bigwedge^r\Omega_{\tilde A,\pp}^1
\]
is locally free of rank $1$ (see \cite[(4.4),(8.7)]{SS72}).
By Corollary~\ref{62}, SNCC for $A$ at $\pp$ is equivalent to $J'_{A,\pp}=C_{\tilde A_\pp/A_\pp}$.
By Lemma~\ref{94}, this is equivalent to $\tilde A_\pp J'_{A,\pp}=J'_{A,\pp}$ and $I_{\tilde A_\pp/A_\pp}=\tilde A_\pp$.
The claim follows using Lemma~\ref{44}.
\end{proof}

\section{Complex analytic spaces}\label{36}

In order to consider analytic spaces, we need in addition to the hypotheses of \S\ref{76} that $k$ is non-discretely valued.
Therefore we assume that $k=\CC$ and consider (germs of) complex analytic spaces.

Let $X$ be a reduced $r$-equidimensional complex analytic space with normalization $\nu_X\colon\tilde X\to X$.
Then there is an $\O_X$-coherent graded $(\Omega_X,d)$-module $\omega_X$ and a trace map 
$c_X\colon\Omega_X\to\omega_X$ (see \cite{Bar78}).
The \emph{Jacobian} and \emph{$\omega$-Jacobian (ideals)} $J_X$ and $J'_X$ of $X$ are defined as in \eqref{72}.
Taking stalks at $x\in X$ leads to the corresponding objects for $A=\O_{X,x}$.
By a \emph{complex analytic singularity} we mean the germ of a complex analytic space.

\begin{dfn}\label{98}
We say that a reduced equidimensional complex analytic space $X$ satisfies \emph{Saito's normal crossing condition (SNCC)} or that $X$ is \emph{free} if $A=\O_{X,x}$ satisfies the corresponding property for all $x\in X$ (see Definition \ref{80} and Definition \ref{100}).
We say that $X$ satisfies a property in codimension (up to) $c$ if it does outside of an analytic subset of codimension at least $c+1$.
We define the corresponding properties for complex analytic singularities by requiring them for some representative.
\end{dfn}

\begin{rmk}\label{38}
That $X$ satisfies SNCC means that the inclusion of coherent $\O_X$-modules $(\nu_X)_*\omega_{\tilde X}^0\into\omega_X^0$ is an equality (see \cite[p.195, Ex.~i)]{Bar78}).
In particular, SNCC is an open condition.

Freeness is an open condition as well.
In fact, Cohen--Macaulay loci of coherent $\O_X$-modules are open (see \cite[Satz~7]{Sch64}) and the Gorenstein locus of a Cohen--Macaulay $X$ is the open set where the coherent $\O_X$-module $\omega_X^r$ is locally free of rank $1$ (see \cite[Thm.~3.3.7.(a)]{BH93}).

Both SNCC and freeness are satisfied in codimension $0$, that is, generically.
\end{rmk}

The following is the analytic version of Proposition~\ref{66}.

\begin{prp}\label{88}
A reduced equidimensional complex analytic singularity $X$ satisfies SNCC if it does in codimension one.
\end{prp}

\begin{proof}
Assume that $X$ satisfies SNCC in codimension one and replace $X$ by a representative.
Let $x\in X$ and set $A:=\O_{X,x}$.
Consider the coherent $\O_X$-module $\F=\omega_X^0/(\nu_X)_*\omega_{\tilde X}^0$ and the coherent $\O_X$-ideal $\I=\ann\F$.
By hypothesis and Remark~\ref{38}, $V(\I)=\Supp\F$ and hence $V(\I_x)$ has codimension at least $2$.
In particular, for any $\pp\in\Spec A$ with $\height\pp=1$, $\ann(\F_x)=\I_x\not\in\pp$ and hence $\omega_{A,\pp}^0/\omega_{\tilde A,\pp}^0=(\F_x)_\pp=0$.
In other words, $A$ satisfies SNCC in codimension one.
Then $\O_{X,x}=A$ satisfies SNCC due to Proposition~\ref{66}.
This means that $X$ satisfies SNCC at $x$.
Therefore $X$ satisfies SNCC as claimed.
\end{proof}

In case of smooth irreducible components our results from \S\ref{40} apply to a transversal curve singularity.

\begin{prp}\label{68}
Let $X$ be a reduced equidimensional complex analytic singularity with smooth local irreducible components in codimension one.
If $X$ satisfies SNCC then it must be a normal crossing divisor in codimension one.
\end{prp}

\begin{proof}
Set $r:=\dim X$ and denote by $m:=n-r$ the codimension of $X$ in some smooth ambient space $(\CC^n,0)$.
We may freely move the base point of the germ $X$ to a general point in codimension one.
Let $Z$ be the reduced singular locus of $X$.
We may assume that $Z\ne\emptyset$ is smooth of codimension one and that the 
irreducible components $X_1,\dots,X_s$ of $X$ are smooth containing $Z$.
By Proposition~\ref{51}, SNCC descends to any union of irreducible components of $X$.
We may therefore assume that $2\le s\le3$ and that $X_1\cup\dots\cup X_{s-1}$ is a normal crossing divisor.
Then there are local coordinates such that
\begin{align}\label{19}
Z&=\{x_1=\dots=x_{m+1}=0\},\\
\nonumber X_i&=\{x_1=\cdots=\wh x_i=\dots=x_{m+1}=0\},\ i=1,\dots,s-1.
\end{align}
By the implicit function theorem, there is a $j\in\{1,\dots,m+1\}$ such that 
\[
X_s=\{x_i=y_i(x_j,x_{m+2},\dots,x_n)\mid j\ne i=1,\dots,m+1\}.
\]
If $y_i\ne0$ then we may write $y_i=x_j^{p_i}u_i$ with $u_i(0,x_{m+2},\dots,x_n)\ne0$.
We may then assume that the latter and hence also $u_i$ is a unit.
Dividing $x_i$ by $u_i$ results in $u_i=1$ leaving \eqref{19} unchanged.
This makes the defining equations of $X_1,\dots,X_s$, and hence of $X$, independent of $x_{m+2},\dots,x_n$.
Then $X$ becomes a product $X=C\times Z$ where $C$ is a curve in the transversal slice $\{x_{m+2}=\cdots=x_n=0\}$.
By Proposition~\ref{3}, with $X$ also $C$ satisfies SNCC.
Then Proposition~\ref{97} forces $C$ to be plane normal crossing.
In particular, $s=2$ and $X$ is a normal crossing divisor.
\end{proof}

\begin{exa}
The free divisor $D=\{xy(x+y)(x+xz)=0\}$ has smooth reduced singular locus $Z=\{x=y=0\}$ and $4$ smooth local irreducible components at points of $Z$.
However it is not analytically trivial along $Z$ in codimension one.
\end{exa}

We are finally ready to prove our main result.

\begin{proof}[Proof of Theorem~\ref{70}]
Suppose first that $X$ satisfies SNCC.
By Proposition~\ref{88}, SNCC for $X$ is a codimension-one condition.
We may therefore assume that $X$ is free and that $\tilde X$ is smooth.
Proposition~\ref{67} then implies that $\nu_X$ is unramified.
By Lemma~\ref{44} this means that $X$ has smooth local irreducible components.
Proposition~\ref{68} then forces $X$ to be a normal crossing divisor in codimension one.
The converse implication follows from Propositions~\ref{3}, \ref{97}, and \ref{88}.
\end{proof}

We conclude with an application of our approach to splayed divisors.
By a \emph{divisor} we mean a reduced hypersurface singularity.
Let $D_1, D_2\subset (\CC^{r+1},0)$ be divisors.
Then $D_1$ and $D_2$ are called \emph{splayed} (see \cite{Fab13}) if 
\[
D_1\cong D_1'\times(\CC^{r_2+1},0),\quad D_2\cong (\CC^{r_1+1},0)\times D'_2.
\]
for divisors $D_i'\subset (\CC^{r_i+1},0)$ for $i=1,2$ under some isomorphism $(\CC^{r+1},0)\cong(\CC^{r_1+1},0)\times(\CC^{r_2+1},0)$.
In this case we call the union $D_1\cup D_2$ a \emph{splayed divisor}.
In other words, splayed divisors are \emph{product unions}
\[
D_1'\prodcup D'_2:=D_1'\times(\CC^{r_2+1},0)\cup(\CC^{r_1+1},0)\times D'_2
\]
of divisors (see \cite[\S3]{Dam96}).
Aluffi and Faber characterized splayedness in terms of logarithmic differential forms (see \cite[Thm.~2.12]{AF13}).
Passing to the residual part of these forms yields a characterization in terms of regular differential forms.

\begin{prp}
Let $D_i=V(h_i)\subseteq (\CC^{r+1},0)$ for $i=1,2$ be divisors.
If $D_1$ and $D_2$ are splayed then the natural map
\begin{equation}\label{77}
\omega^0_{D_1\sqcup D_2}=\omega^0_{D_1}\oplus\omega^0_{D_2}\to\omega^0_D
\end{equation}
is an isomorphism.
The converse holds true if $D=D_1\cup D_2$ is free.
\end{prp}

\begin{proof}
The map in \eqref{77} is obtained using \eqref{99} by applying $\Hom_{\O_D}(-,\omega_D^0)$ to the inclusion 
\begin{equation}\label{16}
\O_{D_1\sqcup D_2}=\O_{D_1}\times\O_{D_2}\infrom\O_D.
\end{equation}

If $D_1$ and $D_2$ have a common irreducible component $D'$, which is not the case if they are are splayed, then applying $\Hom_{\O_D}(-,\omega_D^0)$ to the commutative diagram
\[
\xymat{
\O_{D_1}\times\O_{D_2}\ar@{->>}[d] & \O_D\ar@{->>}[d]\ar@{_(->}[l] & \\
\O_{D'}\times\O_{D'} & \O_{D'}\ar[l]_-{(\id,\id)}
}
\]
and using \eqref{99} yields a commutative diagram
\[
\xymat{
\omega^0_{D_1}\oplus\omega^0_{D_2}\ar[r] & \omega^0_D & \\
\omega^0_{D'}\oplus\omega^0_{D'}\ar@{^(->}[u]\ar[r]^-{+} & \omega^0_{D'}\ar@{^(->}[u]
}
\]
whose top row is \eqref{77}.
As $\omega^0_{D'}\ne0$ this shows that \eqref{77} is not injective in this case.
Therefore we may assume that $D_1$ and $D_2$ do not have a common irreducible component.
Then \eqref{16} has a torsion cokernel and \eqref{77} is an inclusion since $\omega^0_D$ is torsion-free.

As in Proposition~\ref{12} there is a commutative diagram
\begin{equation}\label{79}
\xymat{
\sigma^0_{D_1}\oplus\sigma^0_{D_2}\ar@{^(->}[r]\ar[d]^\cong & \sigma^0_D\ar[d]^\cong\\
\omega^0_{D_1}\oplus\omega^0_{D_2}\ar@{^(->}[r] & \omega^0_D.
}
\end{equation}
In fact, using \eqref{42} and \eqref{16} one computes that 
\[
c_{D_1}+c_{D_2}=\res{dh_1}{h_1}+\res{dh_2}{h_2}\mapsto\res{h_2dh_1+h_1dh_2}{h_1h_2}=\res{d(h_1h_2)}{h_1h_2}=c_D
\]
by the lower inclusion in \eqref{79}.
By \cite[Thm.~2.2]{AF13}, $D_1$ and $D_2$ are splayed if and only if the natural inclusion of Jacobian ideals 
\begin{equation}\label{78}
J_D\into h_2J_{D_1}\oplus h_1J_{D_2}
\end{equation}
is an equality.
Lemma~\ref{46} identifies the upper inclusion in \eqref{79} as the dual of \eqref{78} and the first claim follows.
Indeed, dualizing $\O_{D_1}=\O_D/h_1\O_D$ over $\O_D$ yields
\[
\Hom_{\O_D}(\O_{D_1},\O_D)=\ker(h_1\colon\O_D\to\O_D)=h_2\O_D=h_2\O_{D_1}
\]
and hence by Hom-tensor-adjunction
\[
\Hom_{\O_D}(-,\O_D)
=\Hom_{\O_{D_1}}(-,\Hom_{\O_D}(\O_{D_1},\O_D))
=h_2\Hom_{\O_{D_1}}(-,\O_{D_1})
\]
on $\O_{D_1}$-modules.
Conversely, if $D$ is free then $J_D$ is reflexive and hence
\[
(\sigma^0_D)^{-1}=J_D\into h_2J_{D_1}\oplus h_1J_{D_2}\into h_2\cdot(\sigma^0_{D_1})^{-1}\oplus h_1\cdot(\sigma^0_{D_2})^{-1}=(\sigma^0_{D_1}\oplus\sigma^0_{D_2})^{-1}.
\]
Thus, dualizing an equality in \eqref{79} yields an equality in \eqref{78}. 
\end{proof}

\begin{rmk}
If the divisors $D_1$ and $D_2$ have no common irreducible component then
\[
\tilde D\onto D_1\sqcup D_2\onto D
\]
and condition \eqref{77} can be seen as a weak form of SNCC.
\end{rmk}

\bibliographystyle{amsalpha}
\bibliography{snc}
\end{document}